\documentclass[11pt]{article}
\usepackage[utf8]{inputenc}
\usepackage{mainsty}
\usepackage{geometry}[margin = 1in, top = 1in, bottom = 1in]
\usepackage{authblk}

\title{Fluctuation Bounds for the Restricted Solid-on-Solid Model of Surface Growth}
\author{Timothy Sudijono\footnote{Department of Statistics, Stanford University: 390 Jane Stanford Way, Stanford, CA 94305, USA. Email: \href{mailto:tsudijon@stanford.edu}{\texttt{tsudijon@stanford.edu}} } }
\date{\today}

\begin{document}

\maketitle

\begin{abstract}
The restricted solid-on-solid (RSOS) model is a $(d+1)$-dimensional model of continuous-time surface growth characterized by the constraint that adjacent height differences are bounded by a fixed constant. Though the model is conjectured to belong to the KPZ universality class in dimension one, mathematical progress on the model is very sparse. Motivated by Chatterjee \cite{chatterjee2021superconcentration}, we study basic properties of the model and establish bounds on surface height fluctuations as a function of time. A linear fluctuation bound is proven for all dimensions, and a logarithmic fluctuation lower bound is given for dimension one. The proofs rely on a new characterization of RSOS as a type of first passage percolation on a disordered lattice, which is of independent interest. The logarithmic lower bound is established by showing equality in distribution to the minimum of a certain dual RSOS process.
\end{abstract}

\section{Introduction}

Surface growth models have long been studied in the physics literature as models of forest fire fronts, stain formation, and medium growth. In continuous time, a $d$-dimensional surface growth model is a random function $f(t,x):\bR^+ \times \bb{Z}^d \rightarrow \bb{R}$ which models the surface height of the system. Independent and identically distributed Poisson clocks are assigned to each location in $\bb{Z}^d,$ such that whenever a clock rings, the surface height at that location is updated in some fashion. The Restricted Solid-on-Solid model, introduced in \cite{kim1989growth}, is one such surface growth process.  Whenever a clock at location $x \in \bb{Z}^d$ rings, a block of height one falls and sticks to the surface only if adjacent height differences, after the block falls, are bounded by one. See Figure \ref{fig:RSOS1D} for an example.

The Restricted Solid on Solid (RSOS) model was first introduced in the physics literature as a more realistic model of surface growth by Kim and Kosterlitz \cite{kim1989growth}. Various properties of the model, such as its growth rate and surface roughness were studied numerically. Since then, many papers have studied the model and its variants, with interest driven primarily by the fact that RSOS  is conjectured to belong to the KPZ universality class \cite{park2003universality, kelling2016universality}. Much of this follow-up work has concentrated on providing detailed numerical evidence of this conjecture and related questions, such as the existence of a crossover from random deposition to KPZ behavior \cite{hosseinabadi2010solid, pagnani2013multisurface, chien2004initial, kim2015phase, kelling2016universality}. One question of interest is how the variance of the surface height scales as a function of time. In one dimension, this is believed to be $t^{2/3}.$ For dimensions two and higher, the answer seems to be unclear on account of the extensive numerical simulations done in $\cite{pagnani2013multisurface, kelling2011extremely}$, although earlier works conjectured $t^{1/2}$ for $d=2.$ Other works focus on analytical arguments from physics, for example mean field approximations \cite{barato2007mean} or renormalization group analysis \cite{kloss2012nonperturbative} for the KPZ equation. For further pointers, see the references in \cite{hosseinabadi2010solid}.

Mathematically, the model may be of interest as a generalization of TASEP surface growth to higher dimensions, where adjacent differences may also take values in $\set{-1,0,1},$ or as a variant of stochastic interface models like SOS. To the author's best knowledge, the RSOS surface growth model is distinct from TASEP even in one dimension. Mathematical results on RSOS appear to be few. \cite{chatterjee2021superconcentration} applies a general framework for \textit{discrete-time}, synchronous surface growth driven by Gaussian noise to prove superconcentration in a variant of the RSOS model, and appears to be the first result on rigorous fluctuation upper bounds for any variant of the model, in discrete time. In the same framework, a deterministic version of the model was analyzed in \cite{chatterjee2022convergence}.  Work by Savu \cite{savu2019conserved, savu2004hydrodynamic} studies a continuum RSOS in the hydrodynamic limit, as well as a conserved variant of the model by mean field approximation.

\begin{figure}
    \centering
    \begin{tikzpicture}[scale = 0.8]

    \draw[color=black, fill = lavendergray] (0,1) rectangle ++(1,1);
    \draw[color=black, fill = lavendergray] (1,1) rectangle ++(1,1);
    \draw[color=black, fill= lavendergray] (1,2) rectangle ++(1,1);
    \draw[color=black,fill = lavendergray] (2,1) rectangle ++(1,1);
    \draw[color=black, fill = lavendergray] (3,1) rectangle ++(1,1);
    
    \draw[color=black, fill = applegreen] (0,5) rectangle ++(1,1);
    \draw[color=black, fill = violet] (2,5) rectangle ++(1,1);
    \draw[color=black, fill = bananayellow] (4,5) rectangle ++(1,1);
    
    \draw [-to](0.5,5) -- (0.5,4.5);
    \draw [-to](2.5,5) -- (2.5,4.5);
    \draw [-to](4.5,5) -- (4.5,4.5);
    
    %axis
    \draw (-1,1) -- (6,1);
    \end{tikzpicture}
    \hspace{2mm}
    \begin{tikzpicture}[scale = 0.8]
    
    \draw[color=black, fill = lavendergray] (0,1) rectangle ++(1,1);
    \draw[color=black, fill = lavendergray] (1,1) rectangle ++(1,1);
    \draw[color=black, fill= lavendergray] (1,2) rectangle ++(1,1);
    \draw[color=black,fill = lavendergray] (2,1) rectangle ++(1,1);
    \draw[color=black, fill = lavendergray] (3,1) rectangle ++(1,1);
    
    \draw[color=black, fill = violet] (2,2) rectangle ++(1,1);
    \draw[color=black, fill = bananayellow] (4,1) rectangle ++(1,1);
    
    %axis
    \draw (-1,1) -- (6,1);
    \end{tikzpicture}
    
    \caption{Example of surface dynamics for the RSOS model when the surface height is updated. Blocks fall and stick to the surface only if adjacent height differences remain bounded by one.}
    \label{fig:RSOS1D}
\end{figure}
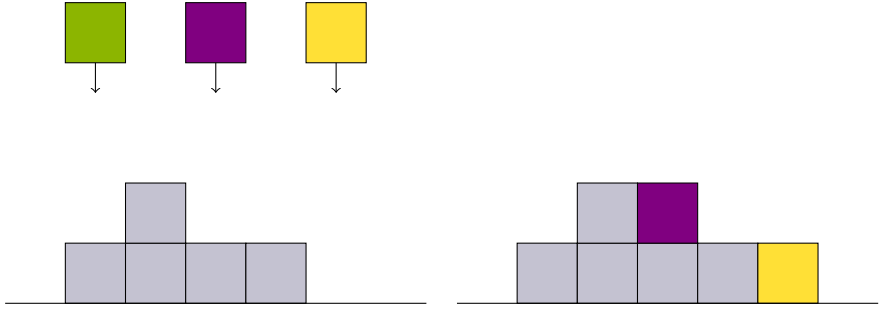

The main contribution of this article is to provide simple bounds on surface height fluctuations of the RSOS model. The primary observation made is that RSOS admits a pathwise representation in terms of paths on a disordered lattice. The space-time locations of the clock rings can be seen as the vertices of a graph contained in $\bR^+ \times \bb{Z}^d,$ sometimes called the Harris construction \cite{ferrari2006harness}. The surface height of the RSOS model can then be expressed as the minimum length of a path on this graph. In this sense, we may interpret RSOS as a type of point-to-line first passage percolation (FPP). These details are provided in Sections \ref{sec:continuoustime_surfacegrowth_prelim}, \ref{sec:rsos}. 

Combined with analysis on the growth rate of the process, the minimal path representation of RSOS is used to prove results on surface height fluctuations in Section \ref{sec:influenceofupdates}. This results in concentration of the RSOS surface height $f(t,x)$ at time $t,$ location $x$:
\[
\Var f(t,x) \leq t,
\]
in all dimensions. To our knowledge, this is the first rigorous bound for the variance of the RSOS model in continuous time, beyond the trivial bound of $O(t^2).$ In one dimension, we show a logarithmic lower bound on the variance:
\[
\liminf \frac{\Var f(t,x)}{\log t} > 0.
\]
The logarithmic lower bound is based on the analysis of a \textit{dual} RSOS process, which is an RSOS process that starts from the well-shaped initial configuration $|x|$. It turns out that the distribution of $f(t,0)$ is equal to that of the minimum of the dual RSOS process at time $t$. The dual process also has a limiting growth rate, which we need to analyze to prove the log lower bound. The overall proof strategy closely mirrors \cite{penrose2008growth}, where a similar lower bound was established for ballistic deposition. We work towards this lower bound in Sections \ref{sec:dual_rep},\ref{sec:dual_growth_rate},\ref{sec:lower_bound}.

\section{Preliminaries on Continuous Time Surface Growth Models}
\label{sec:continuoustime_surfacegrowth_prelim}

Throughout the rest of the paper, fix the following notation.

\begin{enumerate}

    \item $\cl{N}_0$ is the set $\set{\pm e_i}_{i=1}^d \cup \set{0}$ of nearest neighbors to the origin, including the origin itself. Denote $\cl{N} = \set{\pm e_i}_{i=1}^d$ as the set of nearest neighbors without the origin. 
    \item An \textit{update} $\bu$ is a pair $(t,x) \in \bR^+ \times \bR^d$ which denotes the location of a Poisson clock ring in the model. Let $\cl{U}$ denote the Poisson point process driving the surface growth. Informally, the elements of $\cl{U}$ are called updates. Let $\cl{U}_x$ denote the updates which occur over the location $x \in \bb{Z}^d.$
    \item If $\cl{C}$ is a general subset of $\bR_+ \times \bb{Z}^d$ denote by $\pi_x(\cl{C})$ the projection onto $\bb{Z}^d$, representing the locations of the space-time points in $\cl{C}.$ Further, $\pi_t(\cl{C})$ represents the projection onto $\bR_+$, denoting the times of the points in $\cl{C}.$  
    \item $f(t,x)$ is the height of a growing surface. As shorthand, given an update $\bu = (t,x)$ denote by $f(\bu)$ the height of the surface at $(t,x),$ e.g. $f(t,x).$ Let $f(t^-,x)$ denote the left limit 
    \[
    \lim_{ h \downarrow 0} f(t - h,x).
    \]
    Speaking informally, $t^-$ will denote a time infinitesimally before $t$ such that no updates occur in between.
\end{enumerate}

As we will only consider continuous-time surface growth models with integer valued heights in this article, we will often just write \textit{surface growth model} to represent this setting. 

\begin{definition}[Continuous-time surface growth model]
\label{def:surfacegrowth}
A continuous time surface growth model is a function $f(t,x): \bR_+ \times \bb{Z}^d \rightarrow \bb{Z}$ representing the heights of the process at time $t,$ equipped with independent Poisson clocks over each location in $\bb{Z}^d$, and a \textit{driving function} $\phi$, which is a function $\phi: \bb{Z}^{2d+1} \rightarrow \bb{Z}.$ The system evolves as follows: when a clock rings at time $t$, over location $x$, update the value of $f$ by
\[
f(t,x) = \phi( \left(f(t^-,x+n)_{n \in \cl{N}_0} \right)).
\]
\end{definition}

One can also view the surface height $f(t,x)$ as a function of the clock rings, $\cl{U}$ viewed as a Poisson point process on $\bR^+ \times \bb{Z}^d.$ The realization of the clock rings turns out to be a useful object in our study. Throughout this article, we will interchangeably refer to the updates $\cl{U}$ as the \textbf{random (Poisson) lattice}, to emphasize its geometric structure; in the literature it is also known as the Harris graphical construction \cite{ferrari2006harness}. 

The geometric presentation of the updates $\cl{U}$ is useful because one can define a type of random walk on the Poisson lattice. This is similar to random walk representations introduced in \cite{chatterjee2021superconcentration, comets2022scaling, cannizzaro2020brownian} . An analogous construction in continuous time is useful here. See Figure \ref{fig:poissonlattice} for an illustration; horizontal lines segments are drawn at each point on the random lattice to emphasize that lattice path can move laterally upon encountering an update.

\begin{definition}[Path on the random lattice]
\label{def:latticepath}
A \textit{lattice path} ending at $(t,x) \in \bR^+ \times \bb{Z}^d$ is a right continuous function $\gamma(s):[0,t] \rightarrow \bb{Z}^d$, ending at $\gamma(t) = x$, which is piecewise constant and changes value only when $(s,\gamma(s))$ is an update; in this case, $\gamma(s) = \gamma(s^-) + n$ for some $n \in \cl{N}_0.$

We will say a space time point $\mathbf{x} = (t_0,x_0)$ belongs to a path $\gamma,$ denoted by $\mathbf{x} \in \gamma$, if the graph of $\gamma$ goes through $\mathbf{x}$. That is, there exists $s$ such that $(s,\gamma(s)) = (t_0,x_0).$
\end{definition}

\begin{figure}
    \centering
    \includegraphics[width = 0.8\textwidth]{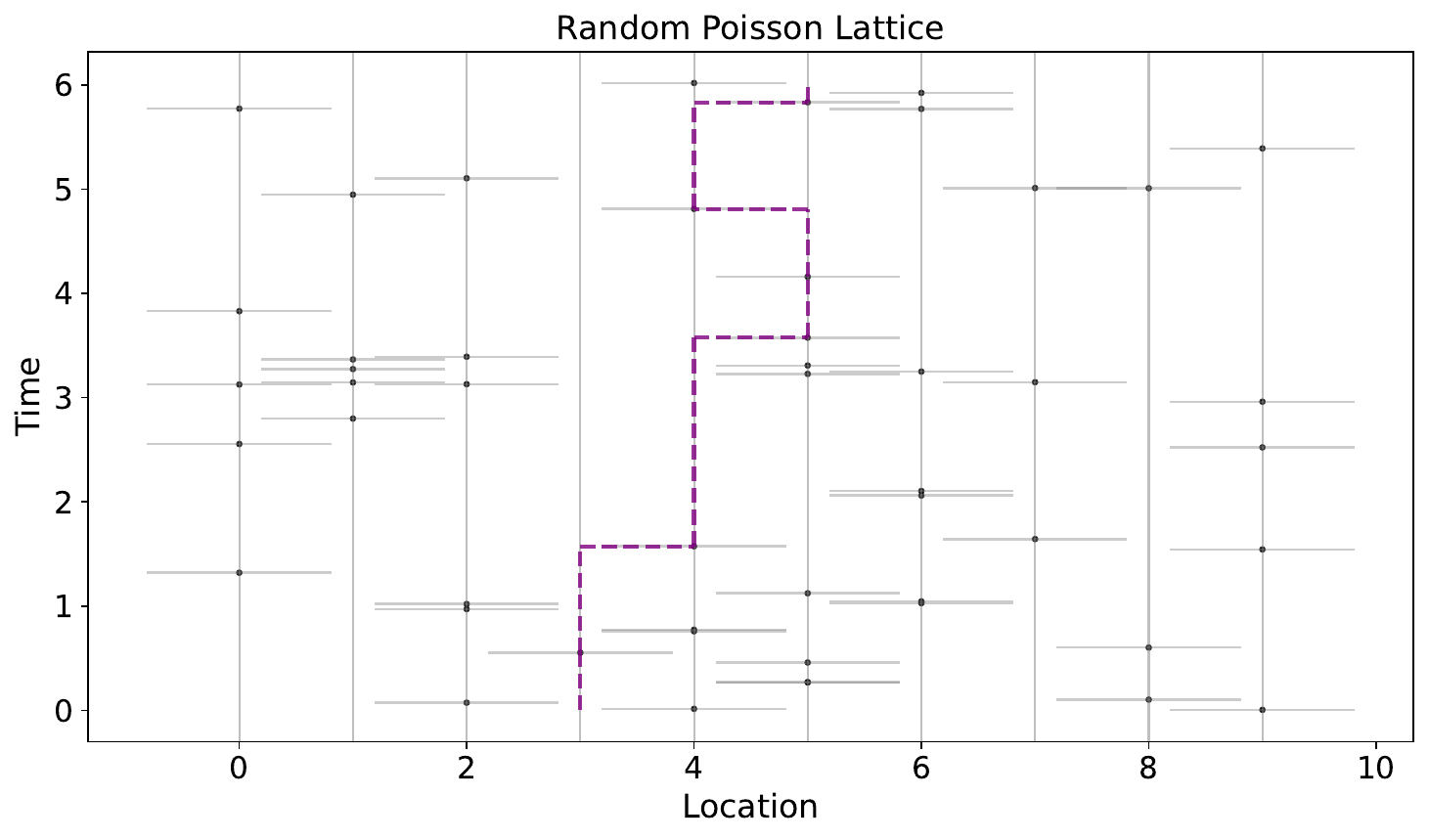}
    \caption{A random Poisson lattice, with an example lattice path dotted in purple. The lattice path can move to neighboring locations only when it encounters an update. It may also stay at the current location.}
    \label{fig:poissonlattice}
\end{figure}

The notion of a path on the random lattice induces a total order on every space-time point in $\bR_+ \times \bb{Z}^d$, called the \textbf{lattice depth}. It can be thought of as a last passage time on the random lattice. The notion is useful in some induction arguments considered later on.

\begin{definition}[Lattice depth of an update]
\label{def:updatedepth}
Fix a random lattice $\cl{U},$ and let $0 < s < t_0$. Define the lattice depth of a space-time point $\bx = (t_0,x_0)$ until time $s$, denoted $D_s(\bx)$, as the supremum of the number of updates $\mathbf{u} = (u_t, u_x)$ such that $\mathbf{u} \in \gamma$ and $u_t \in (s,t_0].$ The supremum is taken over all lattice paths which start at time $s$ and end at $\bx$. For any space-time points before or at time $s$ define the lattice depth as zero.
\end{definition} 

%There may be some concern as to whether the lattice depth of a point may be infinite; the next section will show that we can restrict to a full-measure set where the update process $\cl{U}$ is such that the lattice depth is finite everywhere.

Consider an update $\bu = (t,x)$ with lattice depth until time $s$ given by $D_s(\bu) = d.$ A simple proposition shows that any update at an adjacent location $x+n$ for $n \in \cl{N}$ at a previous time has smaller depth. This will be key in using this ordering for induction arguments.

\begin{proposition}
Consider an update $\bu = (t,x)$ with  finite lattice depth $d.$ Then any update $\bu' = (t',x+n)$ for $n \in \cl{N}_0$ and $t' \in [s,t)$ satisfies
\[
D_s(\bu') < d.
\]
\end{proposition}
\begin{proof}
Consider the path $\gamma'$ which starts at time $s$ and ends at $\bu'$ with the maximal number of updates; the number of such updates is by definition equal to $D_s(\bu').$ The path $\gamma'$ is part of a new path $\gamma,$ which ends at $\bu$, takes value $x+n$ on the time interval $[t',t)$, and agrees with $\gamma'$ on the time interval $[s,t').$ The number of updates on $\gamma$ is at least $D_s(\bu') + 1,$ which establishes $D_s(\bu') < d.$
\end{proof}

\begin{remark}
The lattice depth of any space-time point is actually finite almost surely, and so the supremum over paths in Definition \ref{def:updatedepth} is finite. One can show that $D_0(t,x)$ is itself a surface growth process which satisfies the update rule
\[
D_0(t,x) = 1 + \max_{n \in \cl{N}_0}(D_0(t^-,x+n))
\]
with zero initial condition. This is a ballistic deposition process which is almost surely finite, for all $t,x$. See \cite{penrose2008growth}, Proposition 2.1 and Lemma 4.1. This connection also justifies the intuition that the lattice depth is a type of last passage time.
\end{remark}

\section{The RSOS model of surface growth}
\label{sec:rsos}

In the remainder of this article, we will specialize to the RSOS surface growth model, defined formally as follows.
\begin{definition}
The RSOS model is a surface growth model in the sense of Definition \ref{def:surfacegrowth}, with driving function defined by
\[
\phi\left( (h_n)_{n \in \cl{N}_0} \right) = h_0 + \prod_{n \in \cl{N}} \mathbf{1}\set{h_0 \leq h_n}.
\]
\end{definition}
That is, for any update $\mathbf{u} = (t,x)$, let $f(t,x) = f(t^-,x) + 1$, only if it does not make the differences with adjacent heights greater than one. Otherwise, leave the height unchanged.  Unless otherwise specified, the initial configuration will be the all zero configuration. All subsequent references to a surface growth process with height function $f$ will concern only the RSOS model.

There are two important features about the RSOS model: at all times, neighboring surface heights are bounded by at most one. A simple consequence is that if $f$ denotes the surface height of an RSOS model, then $|f(t,x) - f(t,y)| \leq \norm{x - y}_1$ for all $t$. Secondly, not every update changes the surface height. We distinguish these updates as follows.

\begin{definition}
\label{def:acceptedupdates}

An \textbf{accepted update} on an RSOS surface is an update $\bu = (u_t,u_x) \in \bR^+ \times \bR^d$ which changes the height of the surface:
\[
f(u_t,u_x) = f(u_t^-,u_x) + 1.
\]
Denote the set of all accepted updates by $\cl{A}.$

\end{definition}

\begin{figure}
    \centering
    \includegraphics[width = 0.8\textwidth]{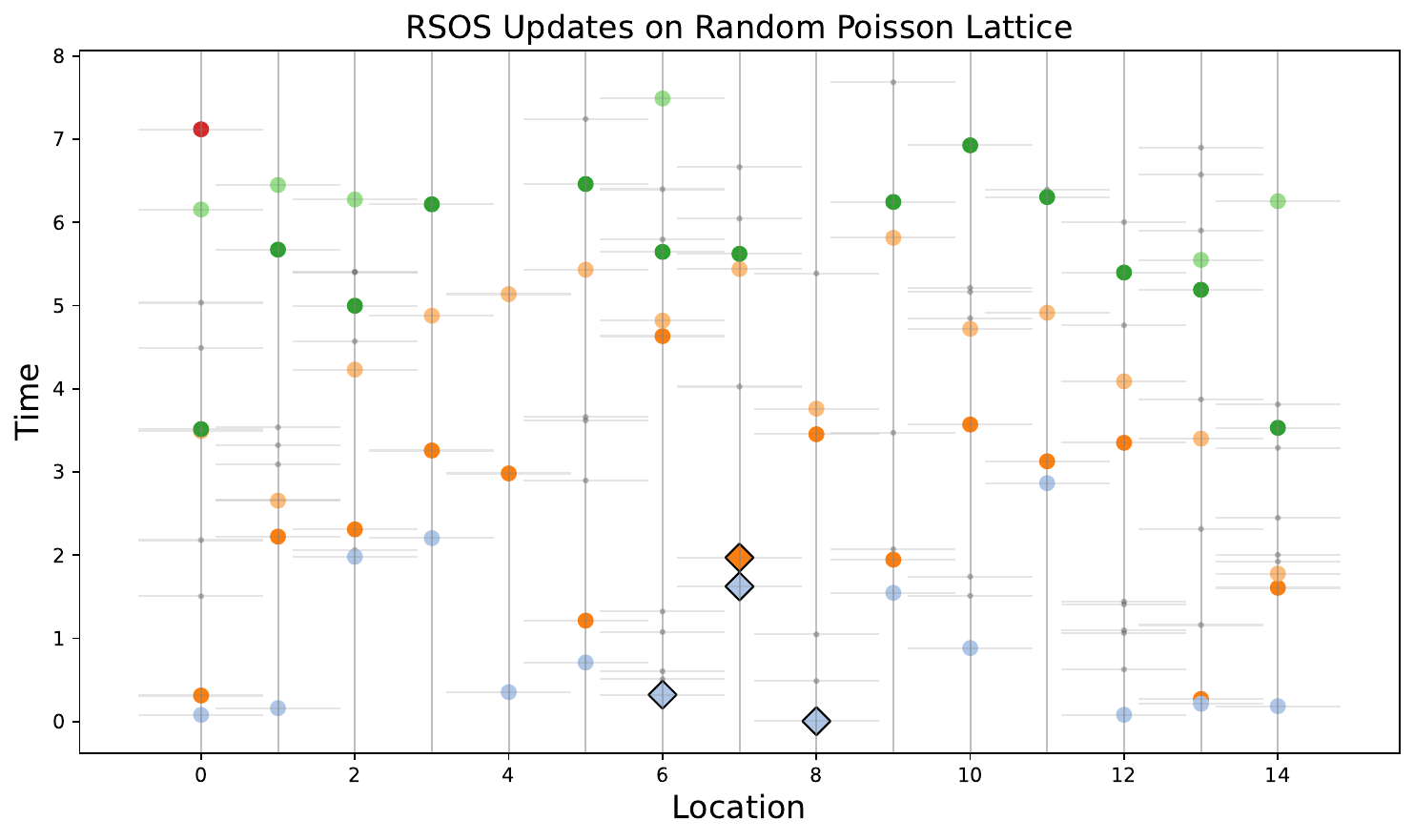}
    \caption{Visualization of accepted updates on the random lattice. The updates are denoted by colored circles, with the colors indicating the new height of the surface. The blue diamond-shaped updates over locations $6,7,8$ indicate the foundation of the orange diamond-shaped update, in the sense of Definition \ref{def:acceptedupdates}.}
    \label{fig:accepted_updates}
\end{figure}

By definition of the RSOS update rule, if there is an accepted update $\bu = (u_t,u_x)$, it must mean that $f(u_t,u_x + n) + 1 \geq f(\bu)$ for all $n \in \cl{N}.$ One way of thinking about this is that for a block to be filled in with height $h$, the blocks at all adjacent locations with height $h-1$ must be present in the surface. The next definition encodes this notion. 

\begin{definition}[Foundation of an accepted update]
Let $\cl{U}$ be the clock ring process associated with an RSOS process $f(t,x)$. Let $\bu = (u_t,u_x) \in \cl{U}$ be an accepted update of the RSOS process $f(t,x),$ such that $f(u_t,u_x) = h,$ with $h > 0.$ Define the \textbf{foundation} $\cl{F}_{\bu}$ of the update $\bu$ as the set of accepted updates $\bv = (v_t,v_x) \in \cl{A}$ such that $v_x - u_x \in \cl{N}_0$ and $f(\bv) = h-1.$ If $h = 0$, $\cl{F}_u = \emptyset.$
\end{definition}

Notice that every update $\mathbf{v}$ in the foundation of $\mathbf{u}$ must satisfy $v_t < u_t$. See Figure \ref{fig:accepted_updates} for an example. The definition also allows us to characterize the accepted updates of an RSOS surface as follows, needing no proof.

\begin{lemma}
\label{lemma:rsosacceptedupdates} 
Let $\bu = (u_t,u_x) \in \cl{A}$ be an accepted update such that $f(\bu) = h.$ Let $\cl{F}_{\bu}$ be the foundation of the accepted update $\bu,$ and let $t_0 = \max \pi_t(\cl{F}_{\bu})$ be the maximum time of all the updates in $\cl{F}_{\bu}$. Then $\bu$ is precisely the first update over location $u_x$ after time $t_0.$ 
\end{lemma}

With this characterization, RSOS updates with height one are the first updates to occur over their locations, when working with the zero initial configuration.

\subsection{Minimal Weight Characterization}

The RSOS model admits a characterization in terms of a certain weight assigned to lattice paths (as in Def. \ref{def:latticepath}). This is the content of Proposition \ref{prop:rsosminchar}, and will be crucial in showing a linear variance bound on surface height. \\

Define $W(\gamma)$, the \textit{weight} of a lattice path $\gamma$, as the number of updates $\bu = (u_t,u_x) \in \cl{U}$ which occur at strictly positive time and such that $\bu \in \gamma$ (recall the notation in Definition \ref{def:latticepath}). This notion can be used to redefine depth in Def. \ref{def:updatedepth}: the depth until time zero is equal to $\max W(\gamma)$ over all paths starting at $t=0$ and ending at $(t,x).$ Finally, denote by $\Gamma_{t,x}$ the set of lattice paths ending at $(t,x)$ and starting at the line $t = 0.$ 

\begin{proposition}
Let $f(t,x)$ be the height function of an RSOS surface, with initial condition $H(x), x \in \bb{Z}^d,$ whose adjacent differences are bounded by one. Then
\[
f(t,x) = \min_{\gamma \in \Gamma_{t,x}} \left[ W(\gamma) + H(\gamma(0)) \right].
\]
For shorthand, define $W_H(\gamma) := \left[ W(\gamma) + H(\gamma(0)) \right].$
\label{prop:rsosminchar}
\end{proposition}

By running the surface growth process from a time $s \in (0,t)$ to $t$ with (random) initial condition equal to the RSOS height at $s$, the minimal path representation implies that 
\begin{equation}
\label{eq:rsos_fpp_formulation}
    f(t,x) = \min_{\gamma \in \Gamma_{(t,x) \rightarrow (s,\bullet)}} \left[ W(\gamma) + f(s,\gamma(s)) \right]
\end{equation}
where $\Gamma_{(s,\bullet) \rightarrow (t,x)}$ denotes the set of all lattice paths starting at time $s$ and ending at $(t,x)$.

To prove Proposition \ref{prop:rsosminchar}, first consider the quantity on the right hand side as a surface growth model of itself. It turns out to have the following update rule:

\begin{lemma}
\label{lemma:wmin_update}
Denote by $W_{\min}(t,x)$ the quantity $\min_{\gamma \in \Gamma_{t,x}} W_H(\gamma).$ Then $W_{\min}(t,x)$ is a surface growth model satisfying the local update rule
\begin{equation}
\label{eq:wmin_surfaceupdate}W_{\min}(t,x) = \min_{n \in \cl{N}_0}\set{1 + W_{\min}(t^-,x+n)}
\end{equation}
with initial condition $H(x).$
\end{lemma}
\begin{proof}[Proof of Lemma]

First, $\min_{\gamma \in \Gamma_{0,x}} W_H(\gamma) = H(x)$ since $\Gamma_{0,x}$ consists of the trivial path starting and ending at $(0,x),$ and the weight function $W(\gamma)$ only counts updates at strictly positive times. Thus, by definition, $W_{\min}(0,x) = H(x).$

Now at every fixed location $x$, it is clear that the quantity $W_{\min}(t,x)$ is a surface growth process, as it only changes at the time of a clock ring. We must prove the local update rule holds for $\bu = (t,x) \in \cl{U}$ an update time. From this observation, the quantity $W_{\min}(t^-,x+n)$ is equal to $W_{\min}(\bu')$ for $\bu'$ the latest update strictly before time $t$ with location $x+n.$ 

Consider every path in $\Gamma_{t,x};$ immediately before time $t$ every path must be at a location $x+n$ for $n \in \cl{N}_0.$ Denote the set of paths (in $\Gamma_{t,x}$) ending at $x+n$ immediately before $t$ by $\cl{P}(t^-,x+n)$ so that 
\[
\Gamma_{t,x} = \bigsqcup_{n \in \cl{N}_0} \cl{P}(t^-,x+n).
\]

Then $\min_{\gamma \in \Gamma_{t,x}} W_H(\gamma) = \min_{n \in \cl{N}_0} \min_{\gamma \in \cl{P}(t^-,x+n)} W_H(\gamma)$. Note that the set $\cl{P}(t^-,x+n)$ is different from the set $\Gamma_{t^-,x+n}$, due to the endpoint of the paths; the difference is that every path in $\cl{P}(t^-,x+n)$ contains the update $\bu.$ Hence,
\begin{equation}
\min_{\gamma \in \Gamma_{t,x}} W_H(\gamma) = \min_{n \in \cl{N}_0} \min_{\gamma \in \Gamma_{t^-,x+n}} 1 + W_H(\gamma),
\end{equation}
which is the desired equality.

\end{proof}

\begin{proof}[Proof of Prop. \ref{prop:rsosminchar}]
We will show that
\[
f(t,x) = W_{\min}(t,x),
\]
Notice that both the left and right hand side are surface growth models. In particular, $W_{\min}(t,x)$ has local update rule given in Eq. \eqref{eq:wmin_surfaceupdate}. It suffices to show that both surface growth models have the same initial condition, and the same update rule. Then they must agree at all space-time points, which can be made formal by an induction on the lattice depth of an update.

The initial conditions agree, as by definition $f(0,x) = H(x)$, and we showed that $W_{\min}(0,x) = H(x)$ in Lemma \ref{lemma:wmin_update}. Now, let $\mathbf{u} = (t,x)$ be an update. We will show that
\begin{equation}
\label{eq:alterante_rsos_update}
f(t,x) = 1 + \min_{n \in \cl{N}_0} f(t^-,x+n),    
\end{equation}
which is exactly the same update rule given in Eq. \eqref{eq:wmin_surfaceupdate}.

Firstly, suppose $\bu$ is an RSOS accepted update, and the surface height changes from $h$ to $h+1$. In particular, $f(t^-,x) = h$ and $f(t^-,x+n) \geq h, \forall n \in \cl{N}_0.$ 
This forces $\min_{n \in \cl{N}_0} f(t^-,x+n) = h$, so both sides of Eq. \eqref{eq:alterante_rsos_update} are equal to $h+1.$

Secondly, if $\bu$ is not accepted, then the RSOS height is unchanged from time $t^-.$ Denote the height $f(t^-,x) = h.$ This implies there exists a $n \in \cl{N}_0$ such that $f(t^-,x+n) = h-1$. Further by the RSOS restriction, $f(t^-,x+n') \geq h-1$ for every $n' \in \cl{N}_0.$ This implies $\min_{n \in \cl{N}_0} f(t^-,x+n) = h-1$, and so both sides of Eq. \eqref{eq:alterante_rsos_update} are equal to $h.$
\end{proof}

As mentioned, the minimum weight characterization established in Prop. \ref{prop:rsosminchar} motivates the consideration of RSOS in a potentially novel way. In \cite{penrose2008growth}, it was shown that the height of ballistic deposition is the maximal path weight. In our setting, the result looks as follows. Consider a ballistic deposition process with `corner' touches. Representing the surface height of the ballistic deposition process by $g(t,x)$, the process has update rule given by
\[
g(t,x) = \max_{n \in \cl{N}_0}(1 + g(t^-,x + n)).
\]
Then 
\[
g(t,x) = \max_{\gamma \in \Gamma_{t,x}} W(\gamma).
\]

This follows similarly to the inductive proofs given for RSOS. In this way however, ballistic deposition can be seen as a directed `last passage percolation' on the random Poisson lattice, viewed as a disordered geometry. In direct analogy, the result established here for RSOS shows that it is an analog of directed first passage percolation on this lattice. \\

Similar representations for surface growth models in terms of paths on the Harris lattice appear in \cite{cannizzaro2020brownian, chatterjee2021superconcentration, ganguly2021cutoff, ferrari2006harness, penrose2008growth, comets2022scaling}.  \cite{cannizzaro2020brownian} studies in great detail a model of surface growth called $0$-ballistic deposition, beginning with the observation that the one-point distribution of the surface height may be described by a type of random walk on a disordered lattice. \cite{ganguly2021cutoff} also provides a representation of Glauber dynamics for the discrete Gaussian free field in terms of functionals for these types of paths; the analysis of the more general harness process \cite{ferrari2006harness} began with a similar observation. In \cite{chatterjee2021superconcentration}, Chatterjee introduced a general method to study fluctuations in discrete time surface growth models, based on a backwards-in-time random walk representation of the variance.

\section{Influence of Updates on RSOS Height}
\label{sec:influenceofupdates}

Throughout this section, we view RSOS surface height as a function of a Poisson point process. To denote the dependence of RSOS height at $(t,x)$ on the update times, write $f_{t,x}(\cl{U})$.\\ 

It turns out that deleting or adding an update on the random Poisson lattice will only affect future surface heights minimally; in fact it will change by at most one. The lower bound is a monotonicity property which will be exploited in other proofs.

\begin{proposition}
\label{prop:rsosheightboundedbysingleupdate}
Suppose there exist two RSOS surfaces $f,\tilde{f}$ which share the same set of updates $\cl{U}$, starting at time $s.$ Suppose that 
\[
f(s,x) \leq \tilde{f}(s,x) \leq f(s,x) + 1, \forall x \in \bb{Z}^d.
\]
Then the same inequality holds for all time $t > s.$
\end{proposition}

\begin{proof}
The proof follows directly from the minimal path formulation Eq. \eqref{eq:rsos_fpp_formulation}, which implies that
\begin{align*}
f(t,x) & = \min_{\gamma \in \Gamma_{(s,\bullet) \rightarrow (t,x)}} \left[ W(\gamma) + f(s,\gamma(s)) \right] \\
\tilde{f}(t,x) & = \min_{\gamma \in \Gamma_{(s,\bullet) \rightarrow (t,x)}} \left[ W(\gamma) + \tilde{f}(s,\gamma(s)) \right]
\end{align*}
for all times $t > s$, $x \in \bb{Z}^d$. Because both surfaces share the same set of updates $\cl{U}$ in $(s,t]$, the set of paths $\Gamma_{(t,x) \rightarrow (s,\bullet)}$ are exactly equal. By the first inequality of the assumption,
\[
W(\gamma) + f(s,\gamma(s)) \leq W(\gamma) + \tilde{f}(s,\gamma(s))
\]
for all paths in $\Gamma_{(s,\bullet) \rightarrow (t,x)}$; taking the minimum over $\gamma \in \Gamma_{(s,\bullet) \rightarrow (t,x)}$ yields $f(t,x) \leq \tilde{f}(t,x)$. Similarly, by the second inequality of the assumption,
\[
W(\gamma) + \tilde{f}(s,\gamma(s)) \leq W(\gamma) + f(s,\gamma(s))+1.
\]
Taking the minimum over $\gamma \in \Gamma_{(s,\bullet) \rightarrow (t,x)}$ yields $\tilde{f}(t,x) \leq f(t,x) + 1$.

\end{proof}

The proof can also be done by directly showing the desired inequality at all accepted updates at times $t > s,$ using an induction on the lattice depth. The question still remains -- if an update is added into the set $\cl{U}$ at location $(s,y)$ to obtain $\cl{U} + \delta_{(s,y)},$ which locations will change the surface height at location $(t,x)$? Call a point $(s,y)$ \textit{influential} if it satisfies
\[
f_{(t,x)}(\cl{U}) + 1 = f_{(t,x)}(\cl{U} + \delta_{(s,y)}).
\]
While it is difficult to understand exactly which updates will affect the height, we can deduce an important necessary condition from the characterization of RSOS as the minimal path weight.

\begin{proposition}
\label{prop:influential_location}
Let $\gamma$ be a path that realizes the minimum weight given the clock ring process $\cl{U}.$ Then if some point $(s,y)$ is influential, $(s,y) \in \gamma.$
\end{proposition}
\begin{proof}
By the characterization in Prop. \ref{prop:rsosminchar}, the RSOS height is equal to the minimum weight over all paths. Suppose that $(s,y) \notin \gamma$. Then $W(\gamma)$ is still unchanged by the addition of the update. This must still be the minimal height, as the addition of an update can never decrease the RSOS height by Prop. \ref{prop:rsosheightboundedbysingleupdate}. Hence $f(t,x)$ is unchanged by the addition of the update at $(s,y).$
\end{proof}

We conclude that the set of influential updates -- those that change the height of the surface -- are contained in a one dimensional subset of the space. From this observation, the linear fluctuation bound follows easily from an appropriate Poincaré inequality. 

\begin{theorem}
\label{thm:linearvariancebound}
In all spatial dimensions,
\[
\Var(f(t,x)) \leq t.
\]
\end{theorem}
The proof uses a Poincaré inequality on Poisson space from \cite{last2017lectures}:

\begin{theorem}[Theorem 18.7 of \cite{last2017lectures}]
Let $\eta$ be a Poisson process on any measurable space $(\bb{X},\cl{X})$ with $\sigma$-finite intensity measure $\lambda.$ Suppose $g \in \bb{L}^2(\Prob_\eta)$, where $\Prob_\eta$ is the distribution of $\eta.$ Then
\[
\Var[g(\eta)] \leq \int \E[(g(\eta + \delta_x) - g(\eta))^2] \lambda(dx).
\]
\end{theorem}

We identify $f_{(t,x)}$ as the function $g$, which is the function of the Poisson process $\cl{U}$ on the space $\bb{X} = \bb{Z}^d \times [0,t]$. Take $\cl{X}$ to be the product sigma algebra. The intensity measure $\lambda$ is given by the sum of Lebesgue measures on each copy of $[0,t]$ which is clearly sigma finite. Note that $f_{(t,x)} \in \bb{L}^2(\Prob_\eta)$ because of the stochastic upper bound proven in Lemma \ref{lemma:poissonub}.

\begin{proof}[Proof of Thm. \ref{thm:linearvariancebound}]
Let $\gamma_{\cl{U}}$ be a minimal path of the clock ring process $\cl{U}$. By Prop. \ref{prop:rsosheightboundedbysingleupdate} and Prop. \ref{prop:influential_location}, bound the difference operator by
\[
f(\eta + \delta_\bu) - f(\eta) \leq \mathbf{1}\set{\bu \in \gamma_{\cl{U}} }.
\]
This yields (by Fubini, given $\sigma$-finiteness of $\lambda$),
\begin{align*}
\Var[f(\eta)] & \leq \E \left[ \int \mathbf{1}\set{\bu \in \gamma_{\cl{U}}} \lambda(d\bu) \right] \\
 & \leq \E \left[ t \right] \\
 & = t,
\end{align*}
as desired.
\end{proof}

The linear in time variance bound is reminiscent of fluctuation bounds for planar first passage percolation models; it mirrors the first few landmark results in this area, such as the linear variance bound of Kesten \cite{kesten1993speed}, and its log-factor improvement by Benjamini, Kalai, and Schramm for a certain distribution of edge weights \cite{benjamini2011first}. See also \cite{benaim2008exponential, damron2014subdiffusive} for extensions of the logarithmic improvement for a much broader class of distributions. \cite{basu2023rotationally} also announces a polynomial improvement to the variance bound via a multiscale argument for certain rotationally invariant FPP models. It would be a natural goal to mirror this line of work and prove a sublinear variance bound for RSOS.

A related open problem is to prove a (sub)-linear variance bound for ballistic deposition. The proof strategy here fails primarily because the analogous upper bound in Prop. 4.1 fails to hold for ballistic deposition.

\section{Bounds on RSOS height}
Having introduced the minimal characterization of the RSOS process and deriving consequences for variance upper bounds, we turn to consequences of this representation for lower bounds on fluctuations, mirroring the arguments of \cite{penrose2008growth}. Ultimately, a logarithmic lower bound on the variance for RSOS surface heights is obtained, in one spatial dimension. Towards this, let us first establish basic observations about the almost sure growth rate of the process which apply in all dimensions.

\begin{lemma}
\label{lemma:poissonub}
Let $f(t,x)$ denote RSOS surface height started from zero initial condition at time $t$, location $x$. Then for any natural number $k,$
\[
\E f(t,x)^k \leq Ct^k,
\]
for some constant $C$ depending on $k.$
\end{lemma}
\begin{proof}
For any $t,x$, view the surface height at $(t,x)$ of the RSOS surface growth model as a function of the underlying Point process. Under this perspective, we claim
\[
f(t,x)(\cl{U}) \leq |\cl{U}_x|,
\]
where the right hand side denotes the number of clock rings over $x$ up to time $t,$ which is $\Poisson(t)$ distributed.

To see this, recall that RSOS only changes height at accepted updates, and at these updates, the height only increases by one. Thus $f(t,x)$ is equal to the number of accepted updates over the location $x.$ Hence $f(t,x)$ is bounded above by the total number updates over the location $x,$ given by $|\cl{U}_x|.$

This inequality provides a coupling between RSOS surface height and a $\Poisson(t)$ distributed random variable. As the $k$th moments of a Poisson$(t)$ random variable are bounded by a multiple of $t^k$, the result follows.
\end{proof}

We provide a matching linear lower bound on the growth rate of RSOS. This uses the minimal path representation developed earlier. To begin, define the event $A$ as the event that there a exists path $\gamma$ from $(t,0)$ to $(0,\bb{Z}^d)$ with weight $W(\gamma) \leq t/(10d).$ The probability of this event is exponentially small:
\begin{lemma}
\label{lemma:exponentially_smallprob_height}
Define
\[
A = \set{ \exists \  \gamma \in \Gamma_{t,0} \text{ s.t. } W(\gamma) \leq t/(10d)}.
\]
Then for $t \geq 1,$ $\Prob(A) \leq \exp(-ct)$ for some positive constant $c$ not depending on $d$.
\end{lemma}
\begin{proof}
Define $A_k$ as the event that there exists a path $\gamma \in \Gamma_{t,0}$ with $W(\gamma) = k$. To bound the probability of $A_k$, notice that there are $(2d+1)^k$ possible ways to choose the locations $x_1,\dots,x_k$ for the updates on the path $\gamma$, and that all of these updates must occur at time less than $t.$ Moreover, other than the $k$ specified updates, no others must occur. As the waiting times between each consecutive update is Exponential$(1)$, and by the memoryless property, the probability that the path in $\Gamma_{t,0}$ with locations $x_1,\dots,x_k$, sees exactly $k$ updates is $\Prob(\text{Poi}(t) = k).$ As a result, $\Prob(A_k) \leq (2d+1)^k\Prob(\text{Poi}(t) = k) \leq (2d+1)^k\Prob(\text{Poi}(t) \leq k)$. \\

The lower tail of the Poisson can be bounded above by Lemma 1.2 of \cite{penrose2003random}, which states that for $x \leq t$,
\[
\Prob(\text{Poi}(t) \leq x) \leq \exp(-t G(x/t))
\]
where $G(x) = 1 - x + x \log x.$ Now, union bound the probability of the event $A$:

\begin{align*}
\Prob(A) & \leq \sum_{k=0}^{t/(10d)} \Prob(A_k)\\
         & \leq \frac{t}{10d}(2d+1)^{t/(10d)}\Prob(\text{Poi}(t) \leq \frac{t}{10d}) \\
         & \leq  \exp\left( \left[\frac{\log(2d+1)}{10d} + \frac{\log 10d}{10d} + \frac{1}{10d} - 1)\right]t  + \log\frac{t}{10d}\right) \\
         & \leq \exp\left( \left[\frac{\log(2d+1)}{10d} + \frac{\log 10d}{10d} + \frac{1}{10d} - 1 \right]t  + \frac{t}{10d} \right)
\end{align*}
which for $t$ large enough is of the form $e^{-c_dt}$ for $c_d = 1 - \frac{\log(2d+1)}{10d} - \frac{\log 10d}{10d} - \frac{2}{10d}  > 0.$ Since the function 
\[
\frac{\log(2d+1)}{10d} + \frac{\log 10d}{10d} + \frac{2}{10d}
\]
is decreasing in $d$ and for $d = 1$ its value is less than $0.5,$ we can lower bound $c_d$ by an absolute constant $c \in (0,1)$ for all dimensions.

\end{proof}

As a simple corollary, this lower bounds the expected value of the RSOS process.

\begin{corollary}
\label{cor:lbev}
Let $f_0(t)$ be the height of an RSOS surface at time $t$ at location $0$. Then, for some $t_0$ depending on $d$, for all $t \geq t_0$
\[
\E[f_0(t)] > C_dt
\]
for some constant $C_d$ depending on dimension.
\end{corollary}
\begin{proof}
As above, let $A$ be the event that there exists path $\gamma$ from $(t,0)$ to $(0,\bb{Z}^d)$ with weight $W(\gamma) \leq t/(10d).$ Note that $\set{f_0(t) \leq t/(10d)} = A$ by the minimal path characterization of RSOS. Using the positivity of $f_0(t)$, 
\begin{align*}
    \E[f_0(t)] & \geq \E[f_0(t)\mathbf{1}_{A^c}] \\
             & \geq \frac{t}{10d}\Prob(A^c) \\
             & \geq \frac{t}{10d}\cdot (1 - e^{-ct}),
\end{align*}
which for $t$ large enough is greater than $t/(20d),$ say.
\end{proof}

\section{Dual Representations of RSOS}

The RSOS process turns out to have a dual representation, in terms of a dual RSOS process started from the initial condition $\norm{x}_1$. The latter is distributionally related to the original RSOS process itself: the minimum of the dual RSOS process has the same distribution as the RSOS surface height at location $0.$ It is interesting that the RSOS process admits this dual characterization, as ballistic deposition has a similar feature: the maximum of the analogous dual process in ballistic deposition has the same distribution as the height of the original process. This is recorded in \cite{penrose2008growth} as Theorem 2.3. Our proposition below mirrors this result.

\label{sec:dual_rep}
\begin{definition}[Dual RSOS process]
Define the RSOS dual process as the surface growth model with height function $f^D(t,x)$, which has initial condition
\[
f^D(0,x) = \norm{x}_1,
\]
and updates in the same way as RSOS.
\end{definition}

Figure \ref{fig:dualrsos} gives an example of the dual RSOS process.

\begin{figure}
    \centering
    \includegraphics[width = 0.9\textwidth]{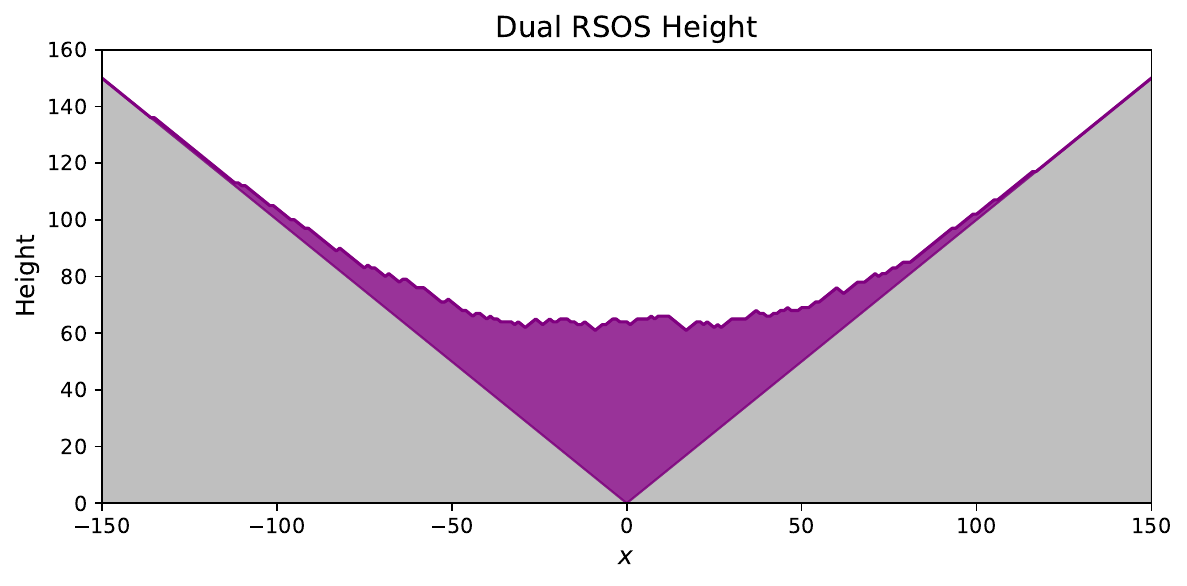}
    \caption{Visualization of surface height of the dual RSOS process, in $1 + 1$ dimensions.}
    \label{fig:dualrsos}
\end{figure}

Associated with a dual RSOS process, we may also consider the set of accepted updates as in Def. \ref{def:acceptedupdates} as those which change the height.

\begin{proposition}
\label{prop:duality}
Let $f(t,x),f^D(t,x)$ denote the height functions of an RSOS and dual RSOS surface respectively, at some fixed $t,x$. Then
\[
f(t,0) \stackrel{(d)}{=} \min_{y \in \bb{Z}^d} f^D(t,y).
\]
\end{proposition}

By translation invariance, distributional equalities for $f(t,y)$ hold for any $y$ in the above statement. The proof proceeds by considering the set of RSOS accepted updates as an object itself. Essentially, the set of RSOS accepted updates defines a kind of `pyramid'; the size of the largest pyramid structure is equivalent to the RSOS surface height. A similar statement holds for dual RSOS and `upside-down' pyramids. Then time inversion establishes a correspondence between these pyramid structures. 

\begin{definition}[Pyramid]
We call a set of updates a \textbf{pyramid} $\cl{P}_{x_0}$ centered at $x_0$ if it is of the form 
\[
\cl{P}_{x_0} = \bigcup_{k = 1}^h L_k
\]
where $L_k$ is a set of updates such that the set of locations $\pi_{x}(L_k)$ is the $\ell_1$- ball in $\bb{Z}^d$ of norm $(h - k) + 1,$ centered at $x_0$. Call $L_k$ the $k$th layer. Moreover, for every $k > 1$ and every $u \in L_k$, $u$ must appear after its neighbors in the previous layer. That is, $\pi_{t}(\bu) > \pi_{t}(\bv)$ for every $\bv \in L_{k-1}$ such that $\pi_x(\bv) - \pi_x(\bu) \in \cl{N}_0.$ The \textit{height} of $\cl{P}_{x_0}$ is equal to $h,$ the number of layers. Say $\cl{P}_{x_0}$ is \textit{within time $t$} if all the updates in $\cl{P}_{x_0}$ occur before time $t$.
\end{definition}

Thus, the last layer of the pyramid has one update, while the penultimate layer of the pyramid as $2d + 1$ updates. Figure \ref{fig:pyramid} gives an example of a pyramid.

\begin{figure}
    \centering
    \includegraphics[width=0.9\textwidth]{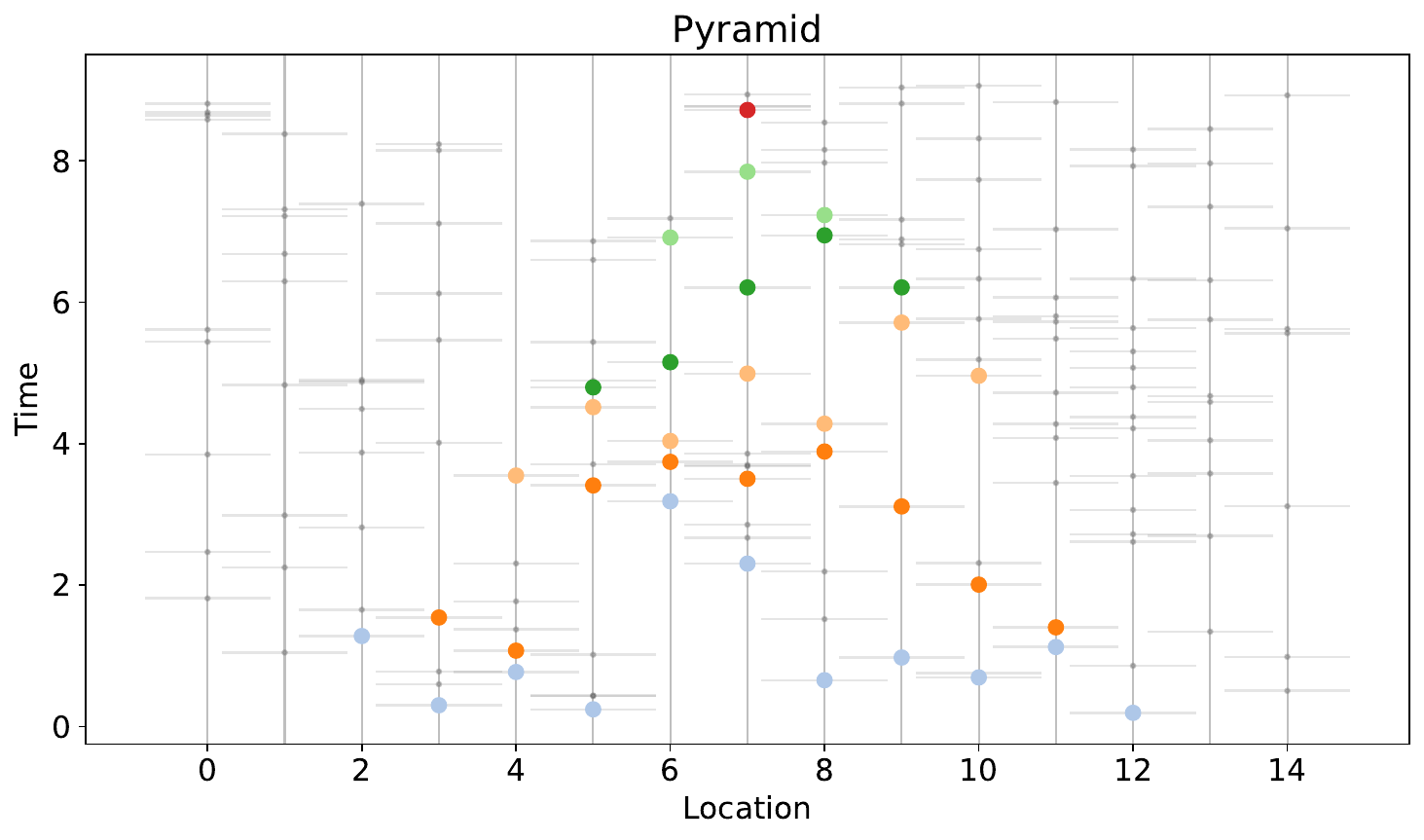}
    \caption{The colored updates indicate a pyramid centered at $x = 7$ with height $6$, which corresponds to the accepted updates of an RSOS process. As it corresponds to the accepted updates of an RSOS process, it also has the maximal height among all pyramids centered at $x = 7$ within time $9$. The colors correspond to the layers of the pyramid.}
    \label{fig:pyramid}
\end{figure}

Similarly, if one looks at the accepted updates of the dual RSOS surface, it forms an ``upside-down pyramid''. 

\begin{definition}
We call a set of updates a \textbf{dual pyramid} within time $T$ if under the time reversal of the random lattice $t \mapsto T - t$, the set forms a pyramid within time $T.$ A dual pyramid also has layers which are defined by the layers of its time reversal. Call the first layer $L_1$ of the dual pyramid the layer with exactly one update.
\end{definition}

The other direction holds true: a pyramid, under time reversal, forms a dual pyramid. This is clear as time reversal is an involution. \\

Let $\scr{P}_{h,x_0}(t)$ and $\scr{P}_{h,x_0}^D(t)$ denote the set of pyramids and dual pyramids respectively of height $h,$ within time $t$, centered at $x_0.$ For the edge cases, define $\scr{P}_{0,x_0}(t) = \scr{P}^D_{0,x_0}(t) = \emptyset$. The equality in distribution relies on the following characterization of RSOS and the dual.

\begin{proposition}
\label{prop:rsospyramidcharacterization}
The height of the RSOS process, with zero initial condition, is equal to the height of the largest pyramid centered at $x_0.$ That is,
\begin{equation}
\label{eq:rsos_maxpyramid} 
f(t,x_0) = \max \set{h: |\scr{P}_{h,x_0}(t)| > 0}.
\end{equation}
Similarly,
\begin{equation}
\label{eq:dualrsos_maxdualpyramid} 
\min_{y \in \bb{Z}^d} f^D(t,y) = \max \set{h: |\scr{P}^D_{h,0}(t)| > 0}.
\end{equation}
If there are no pyramids (resp. dual pyramids) of positive height, then we again interpret the right hand sides of \eqref{eq:rsos_maxpyramid}, \eqref{eq:dualrsos_maxdualpyramid} to be zero.
\end{proposition}

\begin{proof}[Proof of Prop. \ref{prop:rsospyramidcharacterization}]
We first establish Eq. \eqref{eq:rsos_maxpyramid}. Consider first the case where there are no pyramids of positive height. Because of the zero initial condition, the first update over $x_0$ is an accepted update. A single accepted update, by definition, as a pyramid of height one at $x_0$. Thus, when there are no pyramids of positive height, no updates have occurred over $x_0$. Thus $f(t,x_0) = 0.$

\sloppy Consider now the general case. One direction, $f(t,x_0) \leq  \max \set{h: |\scr{P}_{h,x_0}(t)| > 0}$, follows since the set of accepted updates of the RSOS surface is a pyramid of height equal to the surface height $f(t,x_0)$. This can be seen as the condition on every update in the pyramid mirrors the RSOS condition in Lemma \ref{lemma:rsosacceptedupdates}. In particular, construct such a
pyramid by first adding the last accepted
updates over $x_0$ with time $\leq t$. This update comprises the last layer of the pyramid. Then, inductively add the foundations of each update in layer $L_k$ to create layer $L_{k-1}$.

For the converse, consider the set $\scr{P}_{h,x_0}(t)$ for which $h$ achieves the maximum in the equality. Take any $\cl{P}_{x_0} \in \scr{P}_{h,x_0}(t)$. Given $\cl{P}_{x_0},$ there exists another $\cl{P}_{x_0}' \in \scr{P}_{h,x_0}(t)$ (possibly equal to $\cl{P}_{x_0}$) which is contained in the set of accepted RSOS updates. This implies the reverse inequality. 

We accomplish this by iteratively ``pushing down'' each layer of the pyramid $\cl{P}_{x_0}$ in a way that it still satisfies the local RSOS condition. The construction goes as follows. First let $L'_k$ denote the $k$th layer of the pyramid $\cl{P}_{x_0}'$. Define $L_1'$ to be the set of updates $\bu = (u_t,u_x)$ such that $u_x \in \pi_x(\cl{P}_{x_0})$ and $\bu$ is the first update over $u_x.$ Notice that such updates $\bu$ may not be in $\cl{P}_{x_0}.$ \\

Now to iteratively define $L'_k$ for $k \geq 2,$ consider for any update $\bu = (u_t,u_x) \in L_k$ in the original pyramid $\cl{P}_{x_0}$. Let $\cl{G}$ be the set of updates in the previous layer with locations adjacent to $\bu$; that is, the $2d+1$ updates $\bv$ in $L'_{k-1}$ with  $\pi_x(\bv) - \pi_x(\bu)  \in \cl{N}_0$, and let $t_0 = \max \pi_t(\cl{G})$ be the maximum update time over all updates in $\cl{G}$. Define the \textbf{pushdown} map $d_k(\bu)$ to be the first update at location $u_x$ with time greater than $t_0$. Then, iteratively define the next layer of the pyramid as 
\[
L'_k = \set{d_k(\bu): \bu \in L_k}.
\]
Notice that the resulting set is a pyramid, because by construction, each update in a layer $L_k'$ occurs after updates in $L_{k-1}'$ at adjacent locations. Further, the update times of each layer in $\cl{P}'_{x_0}$ are less than that of $\cl{P}_{x_0}$. This shows $\cl{P}_{x_0}'$ has height equal to the height of $\cl{P}_{x_0}.$ Moreover a simple induction on the layer index $k$ shows that $L'_k$ are RSOS accepted updates, using the characterization in Lemma \ref{lemma:rsosacceptedupdates}.  

Next let us establish Eq. \eqref{eq:dualrsos_maxdualpyramid}. Again, consider the case where there are no dual pyramids of positive height. By the dual RSOS initial condition, this is the event that no update has fallen over $0$. Thus $\min_{y \in \bb{Z}^d} f^D(t,y) = 0.$
In the general case, we will show both sides of the inequality. To show 
\[
\min_{y \in \bb{Z}^d} f^D(t,y) \leq \max \set{h: |\scr{P}^D_{h,0}(t)| > 0},
\]
let $h = \min_{y \in \bb{Z}^d} f^D(t,y)$.  The set of accepted updates with height $\leq h$ (those which increase the surface height to a value less than or equal to $h$) is a dual pyramid. Since dual RSOS has initial condition defined by $\norm{x}_1$, it is clear that the set of accepted updates at height $k$ for $k=1,\dots,h$ is the $\ell_1$ ball centered at $0$ of radius $k$.

We must also show that the time reversal of this set is a pyramid; it suffices to show that for every accepted update $\bu$ at height $k$ for $k=1,\dots,h-1$, the set of accepted updates with height $k+1$ with adjacent locations all occur after $\pi_t(\bu).$ But this is enforced by the RSOS surface growth dynamics, as in Lemma \ref{lemma:rsosacceptedupdates}. \\

For the other inequality $\min_{y \in \bb{Z}^d} f^D(t,y) \geq \max \set{h: |\scr{P}^D_{h,0}(t)| > 0}$, apply a similar argument as before. Consider the set $\scr{P}^D_{h,0}(t)$ for which $h$ achieves the maximum in the equality. Take any dual pyramid $\cl{P}_0 \in \scr{P}^D_{h,0}(t)$. Given $\cl{P}_0,$ there exists another $\cl{P}_0' \in \scr{P}^D_{h,0}(t)$ (possibly equal to $\cl{P}_0$) which is contained in the set of accepted dual RSOS updates. Since the layers of the dual pyramid are $\ell_1$-balls, this implies that  
\[
h \leq \min_{y \in \bb{Z}^d} f^D(t,y).
\]

To construct $\cl{P}_0'$, apply the previously-defined pushdown maps $d_k$ layer by layer to $\cl{P}_0.$ It suffices to show that this procedure creates a dual pyramid. As before, this amounts to showing that for every update $\bu$ in $L_k'$, the adjacent updates at the next layer $L_{k+1}'$ occur after $\bu_t.$ But this is enforced by the definition of the pushdown map $d_k.$ A simple inductive argument shows that each layer belongs to the set of accepted updates of the dual RSOS process, using an analog of 
Lemma \ref{lemma:rsosacceptedupdates}.
\end{proof}

\begin{proof}[Proof of the equality in distribution for dual representation]
Using the distributional invariance of the update Poisson process under time reversal $s \mapsto t - s$, note that
\begin{align*}
    f(t,x_0) & = \max \set{h: |\scr{P}_{h,x_0}(t)| > 0} \\
               & \stackrel{(d)}{=} \max \set{h: |\scr{P}^D_{h,0}(t)| > 0}\\  
               & = \min_{y \in \bb{Z}^d} f^D(t,y).
\end{align*}    
The first and third lines follow by the previous lemma, while the second line follows by the symmetry: every pyramid under the time reversal creates a dual pyramid of the same height.
\end{proof}

\section{Dual RSOS Growth Rate}

\label{sec:dual_growth_rate}
The minimum of the dual RSOS process turns out to grow at a linear rate. This analysis is needed to establish the log lower bound on fluctuations in one dimension in the next section. The growth rate being almost surely linear is not a surprise, due to the subadditivity of the process; the main tool here is the Kesten-Hammersley theorem.

\begin{lemma}[Kesten-Hammersley, Thm. 2.9 \cite{smythe2006first}]
\label{lemma:kesten_hammersley}
Let $\set{X_n}_{n \in \bb{N}}$ be a sequence of random variables with distribution functions $F_n$ and finite second moments. Suppose that for each pair $(m,n)$ of positive integers with $m < n$, $F_{n+m} \geq F_m * F_n$, where $*$ denotes the convolution operator. Furthermore, suppose for every such pair $(m,n)$ there exists a random variable $X_{mn}'$ satisfying
\begin{enumerate}
    \item $X'_{mn}$ has distribution function $F_n$.
    \item $X_m$ and $X_{mn}'$ are independent.
\end{enumerate}
Then if  $(X_n)_n$ is a monotone sequence in $n$, there exists a constant $\gamma$ such that 
\[
\frac{X_n}{n} \rightarrow \gamma \text{ a.s.}
\]
\end{lemma}

This result can be extended to random variables $\set{X_{mn}}$ with distributions that are subconvolutive (where Condition (3) above is replaced with $F_{n+m} \leq F_m * F_n$) by considering $\set{-X_{mn}}.$

\subsection{Almost sure limit for the dual RSOS process}
\begin{definition}
As shorthand define the quantities
\begin{align*}
    M_t & := \inf_{x \in \bb{Z}^d} f^D(t,x) \\
    T(u) & := \inf \set{t \geq 0: M_t \geq u}, u \in \bb{N},
\end{align*}
the minimum dual RSOS height and its hitting times.
\end{definition}

\begin{lemma}
\label{lemma:dualrsossubconvolutive}
Let $\set{M_t}_{t \in \bb{R}}$ be the minimum height of a growing dual RSOS process $\set{A(t,x)}_{t,x}$ as above, and let $\set{T(u)}_{u \in \bb{N}}$ be the collection of associated first hitting times. Denote by $F_{T(u)}$ the distribution function of $T(u).$ Then for all $u,v \in \bb{N}$, 
\[
F_{T(u+v)} \leq F_{T(u)} * F_{T(v)}.
\]
\end{lemma}
\begin{proof}
Fix $u,v \in \bb{N}, u \leq v.$ We claim there is a random variable $T^*_u(v)$ such that: 
\begin{itemize}
    \item $T^*_u(v)$ is equal in distribution to $T(v)$
    \item $T^*_u(v)$ is independent of $T(u)$
    \item $T(u+v) \geq T(u) + T^*_u(v)$.
\end{itemize}

To construct this random variable, consider a coupling of the dual RSOS process $\set{A(t,x)}_{t,x}$ and a process, $\set{B(t,x)}_{t,x},$ defined as follows. Let $B(t,x)$ grow in the same manner as $A(t,x)$ up until the hitting time $T(u).$ Let $x_0$ denote the location of the update which was filled in at this time. That is, let $x_0 \in \bb{Z}^d$ be such that $B(T(u),x_0) = u,$ but $B(T(u)^-,x_0) = u-1.$ Now, for all locations on the lattice, increase the heights of $B(T(u),z)$ to be equal to $u + \norm{x_0 - z}_1$. Now after $T(u),$ let $B(t,x)$ grow according to the same update rule as $A(t,x)$. Couple both processes by growing them according to the same Poisson clocks $\cl{U}.$

Intuitively, $B$ follows $A$ up until $T(u),$ and then its height is increased to the same well initial condition as a dual RSOS process, centered at $x_0.$ Then $B$ continues to grow according to the dual RSOS update rule, coupled to the same Poisson clocks as $A.$ 

Let $T^*_u(v)$ be the time after $T(u)$ at which the minimum height of process $B$ first hits $u+v$:
\[
T^*_u(v) = \inf \set{t \geq T(u): \min_{x \in \bb{Z}^d} B(t,x) \geq u+v} - T(u).
\]
Because $B$ grows like a (spatially translated) dual RSOS process after time $T(u),$ starting at height $u$, $T^*_u(v) \stackrel{d}{=} T(v).$ Now, $T^*_u(v)$ and $T(u)$ are functions of Poisson clocks in disjoint sets, so that $T^*_u(v)$ is independent of $T(u)$ by the Strong Markov property \cite{zuyev2006strong}. It remains to be argued that 
%% See https://math.stackexchange.com/questions/807773/strong-markov-property-for-poisson-point-process.
\[
T(u+v) \geq T(u) + T^*_u(v).
\]
At time $T(u)$, upon updating the heights of $B$ to the well configuration, for all $x \in \bb{Z}^d,$ $B(T(u),x) \geq A(T(u),x).$ Indeed, by definition $B(T(u),x_0) = A(T(u),x_0)$, where recall $x_0$ is the location of the update that occurred at time $T(u)$. By the RSOS condition (see the remark before Def. \ref{def:acceptedupdates}), we must have that $|A(T(u),x_0) - A(T(u),x)| \leq \norm{x_0 - x}_1.$ By construction, $B(T(u),x) = B(T(u),x_0) + \norm{x - x_0}_1$, and so the inequality $B(T(u),x) \geq A(T(u),x)$ follows.

Now by the first inequality in Proposition \ref{prop:rsosheightboundedbysingleupdate}, this implies that $B(t,x) \geq A(t,x),\forall x,$ for all $t \geq T(u).$  As a result, the minimum height of $B$ must hit $u+v$ before $A$. By construction, $\min_{x \in \bb{Z}^d} B(s,x) \geq u+v$ at the time $s = T_u^*(v) + T(u),$ while the hitting time of the minimum of process $A$ is by definition $T(u+v).$ Hence $T(u+v) \geq T(u) + T^*_u(v)$ as desired.

This last equation implies the desired claim:
\[
\Prob(T(u+v) \leq x) \leq \Prob(T(u) + T^*_u(v) \leq x),
\]
so that 
\[
F_{T(u+v)} \leq F_{T(u)} * F_{T(v)}
\]
using $T^*_u(v) \stackrel{d}{=} T(v)$ and $T^*_u(v) \perp T(u)$.
\end{proof}

\begin{corollary}
\label{cor:duallineargrowthrate}
There exists a constant $\rho \in (0,\infty)$ such that
\begin{align*}
    & \lim_{u \rightarrow \infty} \frac{T(u)}{u} = \rho \ a.s. \\
    & \lim_{t \rightarrow \infty} \frac{M_t}{t} = \rho^{-1} \ a.s.
\end{align*}
\end{corollary}
\begin{proof}

The first claim follows by the Kesten-Hammersley theorem Lemma \ref{lemma:kesten_hammersley}), provided the assumptions hold. Here, set $X_{m} = T(m)$ and $X_{mn}$ as the hitting time $T^*_m(n)$ associated to the dual RSOS process defined in the proof of Lemma \ref{lemma:dualrsossubconvolutive}. By that argument, $T^*_m(n)$ and $T(m)$ are independent, and $T^*_m(n)$ is equal in distribution to $T(n).$ Since the dual RSOS process $M_t$ is monotone and jumps by at most $1$, $T(n)$ is monotone increasing. 

It remains to show that $\E[T(u)^2] < \infty$ for all $u.$ Consider comparing the minimum of the dual RSOS process $M_t$ to the slower process where each layer must be filled in before updates at the next layer can be accepted. By the well initial condition for dual RSOS, each layer is an $\ell_1$ ball centered at the origin. Thus, each layer $j$ is filled in at a time $G_j$ which is the sum of at most $(2j+1)^d$ independent exponentials of mean $1,$ so the first time the slower process reaches minimum height $u$ is dominated by $\sum_{j=1}^{u} G_j.$ Hence, $T(u)$ is stochastically dominated by $\sum_{j=1}^{u} G_j.$ The sum has finite second moment because each $G_j$ has finite second moment, as $G_j$ is Gamma distributed. 

%Using $\E G_j^2 = \int_0^\infty 2t \Prob(G_j \geq t) dt,$ and the naive union bound $\Prob(G_j \geq t) \leq (2j+1)^d\Prob(\text{Exp}(1) \geq t) = (2j+1)^de^{-t},$ the second moment is finite. \\

We conclude by Kesten-Hammersley (applied to the distributionally subadditive variables $\set{-X_{mn}}$) that there exists $\rho \in [0,\infty]$ such that 
\[
\lim_{u \rightarrow \infty} \frac{T(u)}{u} = \rho \ a.s.
\]

To establish that $\lim_{t \rightarrow \infty} \frac{M_t}{t} = \rho^{-1},$ note that by definition of $T(u)$ as the first hitting time of height $u,$

\begin{equation}
\frac{M_t}{T(M_t)} \geq \frac{M_t}{t} \geq \frac{M_t}{M_t + 1}\cdot \frac{M_t + 1}{T(M_t + 1)}.
\label{eq:sandwichinghittingtime}
\end{equation}

Note that $M_t \uparrow \infty$ almost surely. By sandwiching in equation \eqref{eq:sandwichinghittingtime} we obtain $\lim_{t \rightarrow \infty} \frac{M_t}{t} = \rho^{-1}$ provided $\rho > 0.$ To exclude the case that $\rho = 0,$ the second inequality of Eq. \eqref{eq:sandwichinghittingtime} would imply that $\lim_{t\rightarrow \infty} \frac{M_t}{t} = \infty$ almost surely. However, since $M_t$ is nonnegative, by Fatou's Lemma and duality (Proposition \ref{prop:duality}),
\begin{equation}
\label{eq:limit_dual_growth_is_finite}
\E \left[\lim_t \frac{M_t}{t} \right] \leq \liminf \frac{\E M_t}{t} = \liminf \frac{\E f(t,0)}{t}.
\end{equation}
Thus, the upper bound of Lemma \ref{lemma:poissonub} gives that $\E \left[\lim_t \frac{M_t}{t} \right] < \infty$ and so $\rho \neq 0$.

To exclude the $\rho = \infty$ case, Lemma \ref{lemma:exponentially_smallprob_height} and duality again shows that $M_t/t < 1/10d$ with exponentially small probability in $t.$ A simple argument by Borel Cantelli shows that $\Prob(M_t/t < 1/10d \text{ i.o.}) = 0,$ so that $\lim \frac{M_t}{t} \geq 1/10d$ almost surely.
\end{proof}

\section{Lower bounds on surface height fluctuation}

\label{sec:lower_bound}
We will combine our previous results on the properties of the dual RSOS process and the equality in distribution to obtain results on the fluctuation of surface height in one dimension. 

A similar result to the log lower bound is established in \cite{penrose2008growth} for the ballistic deposition model. The idea in \cite{penrose2008growth}, which we follow closely, is inspired by a paper of Pemantle and Peres on log lower bounds for planar first passage percolation, assuming exponential weights \cite{pemantle1994planar}. Proving fluctuation lower bounds on FPP-related problems is an area of active interest; see the recent article of Bates and Chatterjee \cite{bates2020fluctuation} as an example.

\begin{theorem}
Let $f_0(t)$ denote the height of an RSOS surface at zero, at time $t,$ from zero initial condition. Then
\[
\liminf_{t \rightarrow \infty} \frac{\Var(f_0(t)) }{\log t} > 0.
\]
\label{thm:loglowerbound}
\end{theorem}
Because of the equality in distribution of RSOS and the minimum of the dual RSOS process, we can work with $M_t$ instead of $f_0(t).$ Thus, throughout this section let us focus on the dual RSOS process. Firstly, define the \textit{(dual) RSOS interface} at time $t$ as the set of locations in $\bb{Z}^d$ which have an accepted update up to time $t$. To be clear, if $f^D(t,x)$ denotes the dual RSOS height at time $t$, the RSOS interface is the set of all locations $x$ such that $f^D(t,x) > \norm{x}_1.$ Due to the initial condition of the dual RSOS process, in one dimension this set is almost surely a bounded interval.

Call an update \textit{interface-accepted} if the location of the update belongs to the RSOS interface at the time of update. In contrast to Definition \ref{def:acceptedupdates}, these updates may not affect the height of the process. Define $\cl{F}$ as the sigma algebra generated by the locations in $\bb{Z}$ of all interface-accepted updates. With $T(u)$ as the first hitting time of height $u$, let $A(u)$ denote the number of interface-accepted updates up to this time $T(u)$; in general let $N_t$ denote the number of interface-accepted updates until time $t$ so that $A(u) = N_{T(u)}.$ Enumerate the arrival times of the interface-accepted updates as $\tau_1,\tau_2,\dots.$ 

In one dimension, asymptotics for these quantities may be established. Let $I_t$ be the length of the RSOS interface as an interval at time $t,$ and let $Y_n = I_{\tau_n}$ be the width of the interface after $n$ interface-accepted updates. It is not hard to see that the right and left edges of the interface grow at rate $1$, due to the RSOS condition and the well initial configuration. As a result, $I_t$ is a renewal process with Exponential rate $2$ interarrival times. By the strong law of large numbers for renewal processes (see e.g. Thm. 4.4.1 of \cite{durrett2019probability}) there is a positive constant $\gamma$ such that
\begin{equation}
\label{eq:I_t_slln}
\lim_{t\rightarrow \infty}(I_t/t) = \gamma, \quad \text{ a.s. }
\end{equation}

Next, note the following asymptotic scalings. As in \cite{penrose2008growth}, Eq. (4.12), we have
\begin{equation}
    \lim_{t \rightarrow \infty} \frac{N_t}{\int_0^t I_u du} = 1 \text{\ a.s.}, 
\end{equation} 
as $N_t$ is a Poisson process with intensity $I_t$, changing over time. Intuitively, the change in $I_t$ is independent of $N_t$, as it depends on update times immediately to the left and right of the RSOS interface. Combining the previous two observations yields
\begin{equation}
\label{eq:number_arrivals_scaling}
    \lim_{t \rightarrow \infty} \frac{N_t}{t^2} = \frac{\gamma}{2}. 
\end{equation}
By definition $N_{\tau_j} = j,$ so combined with Equations \eqref{eq:I_t_slln}, \eqref{eq:number_arrivals_scaling}, $Y_j \sim \gamma\tau_j \sim (2\gamma j)^{1/2}$ as $j \rightarrow \infty$ almost surely. Similarly,
\[
A(u) = N_{T(u)} \sim (\gamma/2)T(u)^2 \sim (\gamma\rho^2/2)u^2,
\]
which follows by the almost sure growth rate of the hitting time from Corollary \ref{cor:duallineargrowthrate}. $\rho$ is the constant in the same corollary.

To proceed, the crucial observation - as in \cite{pemantle1994planar} - is the following.
Conditional on $\cl{F}$ the distribution of $T(u)$ is the same as that of the sum of $A(u)$ independent exponentials with the $j$th exponential having mean $Y_{j-1}^{-1}.$ This follows by the memoryless property of the exponential and the fact that the minimum of $n$ rate $1$ exponentials is exponential with rate $n.$ Using a quantitative CLT, it can be shown that the conditional distribution of $T(u)$ is close to normal with variance $\sim \log u.$ These details overlap with those in the proof of Theorem 2.2 of \cite{penrose2008growth}. Only the final step is significantly different and uses properties of the RSOS process.

\begin{proof}[Proof of Theorem \ref{thm:loglowerbound}]

By the Berry-Esseen Theorem, there exists a constant $C_{BE} > 0$ such that if $X_1,\dots,X_k$ are independent mean zero, finite third moment random variables with $W := \sum_{i=1}^k X_i$ having variance $1,$ then
\[
\sup_{x \in \bR} \left| \Prob(W \leq x) - \Phi(x) \right| \leq C_{BE} \sum_{i=1}^k \E[|X_i|^3].
\]

Let $e_i := \tau_i - \tau_{i-1}$ and let $\theta_u$ denote the third moment sum 
\[
\theta_u := \sum_{i=1}^{A(u)} \E[|e_i - \E[e_i | \cl{F}]|^3 | \cl{F}].
\]
By properties of the exponential, this can be calculated as $\theta_u = (12e^{-1} - 2)\sum_{i=1}^{A(u)} Y^{-3}_{i-1}.$ This is noted in Equation 4.17 of \cite{penrose2008growth}. 

Applying the Berry Esseen theorem,
\[
\sup_{x \in \bR} \left| \Prob\left[ \frac{T(u) - \mu_u}{\sigma_u} \leq x | \cl{F} \right] - \Phi(x) \right| \leq \frac{C_{BE}\theta_u}{\sigma_u^3},
\]
with $\sigma_u^2 = \Var(T(u)| \cl{F}) = \sum_{j=1}^{A(u)}Y_{j-1}^{-2},$ and $\mu_u = \E[T(u) | \cl{F}].$ The scalings determined above for $A(u)$ and $Y_j$ imply the following:
\begin{align*}
    \sigma_u^2 & \sim (2\gamma)^{-1}\log A(u) \sim \gamma^{-1} \log u \\
    \mu_u & \sim (2\gamma)^{-1/2}\sqrt{A(u)} \sim \rho u/2.
\end{align*}    
By a similar argument, $\theta_u$ is finite almost surely as $\sum_{i=1}^\infty Y_{i-1}^{-3}$ converges to a constant. Now because $\theta_u$ is finite and $\sigma_u^3 = O((\log u)^{3/2}) \uparrow \infty,$ there exists $u_0$ such that for every $u \geq u_0,$
\begin{align}
\label{eq:goodsetprobbound}
    & \Prob(A_u) \leq 0.01, \text{ where }  \\
    A_u & = \set{\frac{C_{BE}\theta_u}{\sigma_u^3} > 0.01} \cup \set{\frac{\sigma_u^2}{\gamma^{-1}\log u} \in (0.5,2)^c }. 
\label{eq:berryesseengoodset}
\end{align}

Combining the Berry-Esseen bound and the definition of $A_u$, 
\begin{equation}
\Prob(T(u) \leq y + 0.2(\gamma^{-1}\log u)^{1/2}) - \Prob(T(u) \leq y) < 1/4,
\label{eq:loglowerbound_proof_I}    
\end{equation}

as in Eq. 4.19 of \cite{penrose2008growth}. Now, let $\nu(t)$ be a median of $M_t$ for $t \geq 0.$ Then $\Prob(T(\nu(t)) > t) = \Prob(M_t < \nu(t)) \leq 1/2$, which shows that $\Prob(T(\nu(t)) \leq t) \geq 1/2.$ From Eq. \eqref{eq:loglowerbound_proof_I}, for $t$ large enough such that $\nu(t) > u_0$, notice that 
\[
\Prob(T(\nu(t)) \leq t - 0.2 \sqrt{\gamma^{-1}\log \nu(t)}) \geq 1/4.
\]
Because of the almost sure convergence $M_t/t \rightarrow \rho^{-1}$, $\nu(t)/t \rightarrow \rho^{-1}$ deterministically, and so $\log \nu(t)$ behaves asymptotically like $\log t$ as $t \rightarrow \infty.$ The event that $T(\nu(t)) \leq t - 0.2 \sqrt{\gamma^{-1}\log \nu(t)}$ is equivalent to  $M_{t - 0.2 \sqrt{\gamma^{-1}\log \nu(t)}} \geq \nu(t)$ . \\

Our next step will be to show that 
\begin{equation}
\label{eq:theorem_8.1_subgoal}
    \Prob\left(M_t \geq \nu(t) + C\sqrt{\log t} \right) \geq 1/8.
\end{equation} 

If so, the definition of the median gives the bound $\Prob(M_t \leq \nu(t)) \geq 1/2$, while Eq. \eqref{eq:theorem_8.1_subgoal} shows $\Prob \left(M_t \geq \nu(t) + C\sqrt{\log t} \right) \geq 1/8$. Thus the fluctuations of $M_t$ are at least of log order, for $t$ large enough. Because $M_t$ and the RSOS height at zero have the same law, this gives the desired logarithmic lower bound on the variance. \\

Towards this goal, notice we may restrict the minimum in the definition of $M_t$ to a much smaller interval. If $[l_t,r_t]$ denotes the RSOS interface, notice that $f^D(t, l_t) \leq f^D(t, y)$ for all integers $y \leq l_t$. Similarly, $f^D(t, r_t) \leq f^D(t, z)$ for all integers $z \geq r_t$. Thus, $M_t := \min_{x \in \bb{Z}^d} f^D(t,x) =\min_{x \in [l_t,r_t]} f^D(t,x).$ As noted earlier, the left and right endpoints grow at rate $1$, so that $l_t,r_t$ are distributed as Poisson with mean $t$. Therefore, for $t$ large enough, on a set $A_1$ of exponentially small probability, $[l_t,r_t] \not \subseteq [-t^2,t^2]$. On the same set, 
\[
M_t \neq \min_{x \in [-t^2,t^2]} f^D(t,x).
\]

Let us use $s$ as shorthand for the time quantity $t -0.2\sqrt{\gamma^{-1}\log \nu(t)}$. Below, let $\Gamma_{x,y}$ denote all lattice paths from starting at $(s,y)$ and ending at $(t,x)$. Then on $A_1^c$, by Proposition \ref{prop:rsosminchar} and the following comment,
\begin{align*}
    M_t & =  \min_{x \in [-t^2,t^2]} f^D(t,x) \\
    & = \min_{x \in [-t^2,t^2]} \left( \min_{y \in \bb{Z}^d} \set{f^D(s,y) +\min_{\gamma \in \Gamma_{x,y}}W(\gamma) } \right) \\
    & \geq \min_{x \in [-t^2,t^2]} \set{M_s + \min_{y \in \bb{Z}^d} \min_{\gamma \in \Gamma_{x,y}} W(\gamma)} \\
    & = M_s + \min_{x \in [-t^2,t^2]} Z(x)
\end{align*}
where we have used the simple inequality $\inf_{i \in I} (a_i +b_i) \geq \inf_{i \in I} a_i + \inf_{i \in I} b_i$. Here, $Z(x) = \min_{y \in \bb{Z}^d} \min_{\gamma \in \Gamma_{x,y}} W(\gamma)$. By the minimum path formulation, $Z(x)$ is equal in distribution to an RSOS surface growth process, measured at time $t-s = 0.2\sqrt{\gamma^{-1}\log \nu(t)}$. Using a union bound with Lemma \ref{lemma:exponentially_smallprob_height}, there is a set $A_2$ with exponentially small probability on which $\min_{x \in [-t^2,t^2]} Z(x) \leq \frac{0.2}{10d}\sqrt{\gamma^{-1}\log \nu(t)}.$ 

Let $A_3 = \set{T(\nu(t)) \leq t - 0.2 \sqrt{\gamma^{-1}\log \nu(t)}}$. For $t$ large enough so that $\log \nu(t) \geq \frac{1}{2}\log{t}$, on $A_1^c \cap A_2^c \cap A_3$,  $M_t \geq \nu(t) + C\sqrt{\log t}$. The probability of $A_3$ is bounded below by $1/4$, and since $\Prob(A_1^c \cap A_2^c) \rightarrow 1$, we conclude that for some $t$ large enough, 
\[
\Prob(M_t \geq \nu(t) + C\sqrt{\log t}) \geq 1/8, 
\]
as desired, which is exactly Eq. \eqref{eq:theorem_8.1_subgoal}.

\end{proof}

It seems possible to extend this to result to higher dimensions as well, although one would get a constant order fluctuation lower bound. The proof would require an understanding of the growth of the dual RSOS interface in higher dimension, which has the same dynamics as the Eden growth model. Limit shape results are well-known for the Eden model, which may be enough to extend the lower bound to all dimensions. The following simple observation makes a start. Ideally, we would like to prove that the liminf of the variance is bounded away from zero, as in Theorem 3.2 of \cite{penrose2008growth}, but the same proof technique fails for RSOS. The basic difficulty is that an update may not change the height of the surface.

\begin{proposition}
Let $f(t,x)$ denote the surface height of a RSOS surface growth model on $\mathbb{Z}^d$. Then
\[
\limsup_{t \rightarrow \infty} \Var f(t,x) > 0.
\]
\end{proposition}

\begin{proof}

Suppose for the sake of contradiction that for some time $t,$ $\Var(f(t-1,x)) < \e$ for some $\e > 0$ small enough. By translation invariance of the law, this implies that $\Var(f(t-1,y)) < \e$ for all $y$. Thus, $\E[(f(t-1,y) - f(t-1,z))^2] \leq 4\e$ for any $y,z \in \bb{Z}^d$. 

To see this,
\begin{align*}
    & \  \E[(f(t-1,y) - f(t-1,z))^2] \\
    = & \  \E[\left(f(t-1,y) - \E[f(t-1,0)] - (f(t-1,z) - \E[f(t-1,0)])\right)^2] \\
    = & \Var(f(t-1,y)) + \Var(f(t-1,z)) - 2\operatorname{Cov}(f(t-1,y), f(t-1,z)) \\
    \leq & \  4\e,
\end{align*}
using the variance bound, translation invariance, and Cauchy-Schwarz.

Thus
\begin{equation}
\label{eq:constantvariance_l2bound}
\E[\sum_{y \sim 0} (f(t-1,y) - f(t-1,0))^2] \leq 8d\e.
\end{equation}
Let $A$ denote the event that at time $t-1$, the height of the process at zero equals the height of that of its neighbors. Notice the integrand on the left hand side of  \eqref{eq:constantvariance_l2bound} is integer-valued. Thus $\E[\sum_{y \sim 0} (f(t-1,y) - f(t-1,0))^2] \geq \Prob(A^c)$ and so $\Prob(A) \geq 1 - 8d\e$.

Now, note that on the event $A,$ the height of the process satisfies $f(t,0) = f(t-1,0)$ if and only if there is no update that falls over location zero between times $[t,t-1)$; meanwhile $f(t,0) > f(t-1,0)$ if and only if there is at least one update that falls over location zero between these times. These two possibilities happen with constant probabilities $p,1-p$ respectively. By properties of Poisson point processes, the update times before time $t-1$ are independent of those in the time interval $[t-1,t).$ As a result, on the event $A,$  $\Prob(f(t,0) = f(t-1,0) | \cl{F}_{t-1}) = p$ while $\Prob(f(t,0) > f(t-1,0) | \cl{F}_{t-1}) = 1-p$, where $\cl{F}_{t-1}$ is the sigma algebra generated by the Poisson noise up until time $t-1.$

We may use this to bound the conditional variance from below. For shorthand, write $f_t := f(t,0)$ for all $t.$ Since $f_t$ is monotone increasing and integer valued,
\begin{align*}
    \Var( f_t | \cl{F}_{t-1}) & = \E[(f_t - \E[ f_t|\cl{F}_{t-1}])^2\mathbf{1}_{f_t = f_{t-1}}| \cl{F}_{t-1}] \\
    & + \E[(f_t - \E[f_t|\cl{F}_{t-1}])^2\mathbf{1}_{f_t \geq f_{t-1} + 1}| \cl{F}_{t-1}].
\end{align*}
Now if $\E[f_t | \cl{F}_{t-1}] \geq f_{t-1} + \frac{1}{2}$ the above decomposition shows that $\Var(f_t|\cl{F}_{t-1}) \geq \frac{1}{4}\Prob(f_t = f_{t-1}|\cl{F}_{t-1})$ because the first term above is bounded by the same quantity. If instead $\E [f_t | \cl{F}_{t-1}] < f_{t-1} + \frac{1}{2}$, then the second term above is bounded below by $\frac{1}{4}\Prob(f_t \geq f_{t-1} + 1|\cl{F}_{t-1})$. Combining these two facts,
\begin{align*}
    \Var(f_t|\cl{F}_{t-1})\mathbf{1}_A & \geq \frac{1}{4}\min\left(\Prob(f_t = f_{t-1}|\cl{F}_{t-1}),\Prob(f_t \geq f_{t-1} + 1|\cl{F}_{t-1})\right)\mathbf{1}_A \\
    & = \frac{1}{4}\min\left(\Prob(f_t = f_{t-1}|\cl{F}_{t-1})\mathbf{1}_A,\Prob(f_t \geq f_{t-1} + 1|\cl{F}_{t-1})\mathbf{1}_A\right) \\
    & = \frac{1}{4}\min(p,1-p)\mathbf{1}_A.
\end{align*}

Thus, $\E[\Var(f_t|\cl{F}_{t-1})] \geq \frac{1}{4}\min(p,1-p)\Prob(A)$, so $\Var(f(t,0))$ is also bounded away from zero by the law of total variance. As a result, whenever $\Var(f(t-1,x))$ is small enough, the variance at time $t$ is bounded away from zero, establishing that the limsup must be positive.

\end{proof}

\section{Discussion}
\label{sec:discussion}

Since the RSOS model appears to be non-integrable, probing the fine structure of the model is difficult. This paper makes a start by offering a characterization of the RSOS surface height as the minimal weight path on the disordered Poisson lattice, as well as a characterization in terms of a dual process. These observations are used to obtain linear fluctuation upper bounds for the model in all dimensions, and a log lower bound in one dimension. 

Several extensions of the RSOS model are also of interest. In the physics literature, variants have been studied where adjacent heights are bounded by $k > 1$. The methods in this paper should extend to these models as well. For example, the minimum weight characterization of Prop. \ref{prop:rsosminchar} holds if we assign the path weight $W_k(\gamma) = |\set{\text{vertical } u \in \gamma}| + k|\set{\text{horizontal } u \in \gamma}|$, where a horizontal update is one such that the path $\gamma$ moves to another location upon passing through, and a vertical update is an update such that the path does not switch locations. The proof is obtained by a comparison to the surface growth model with update rule
\[
f(t,x) = \min( f(t^-,x) + 1 , \min(\set{k + f(t^-,x+n)}_{n \in \cl{N}})).
\]
Following our remark on the interpretation of RSOS and ballistic deposition as analogues of first and last passage percolation on this random geometry, a more general class of directed polymer models on the random Poisson lattice might be of interest. That is, one can place a Gibbs measure on weights proportional to $\exp(\beta W(\gamma)),$ where intuitively $\beta = \pm \infty$ correspond to the restricted solid on solid model and ballistic deposition. The general $\beta$ case bears a strong resemblance to a model introduced in \cite{cannizzaro2020brownian} to understand the large scale limit of ballistic deposition. These could be objects of future study.

In general, tools to understand surface height fluctuations for continuous-time surface growth models are still under-developed. A natural goal for future research is to extend \cite{chatterjee2014superconcentration} to the continuous-time, asynchronous update setting, though a superconcentration theory in this setting seems to require more assumptions on the surface growth dynamics. Beyond understanding surface height fluctuations, studying scaling limits of the model is also of interest. For example, in recent work of \cite{comets2022scaling}, a variant of ballistic deposition was shown to have a (non-KPZ) scaling limit in terms of a continuum LPP; by analogy, one could investigate the scaling limit of the RSOS model in terms of first passage percolation.

\section{Acknowledgments}

The author thanks Sourav Chatterjee for suggesting this problem as well as helpful conversation, guidance,
and encouragement. The author also thanks Erik Bates, Will Hartog, Kevin Guo, and two anonymous referees for helpful comments. The author acknowledges support from the NSF Graduate Research Fellowship Program under Grant DGE-1656518.

\bibliographystyle{alpha}
\bibliography{main}

\end{document}